\documentclass[reqno, 11pt]{amsart}
\usepackage{amsmath,amsthm,amssymb,bbm,color,verbatim,tikz}\usetikzlibrary{matrix}

\input xy 
\xyoption{all}

\numberwithin{equation}{section}

\usepackage{hyperref}
\hypersetup{colorlinks=true, linkcolor=blue, citecolor=blue}

\newcommand{\CARD}{{\rm CARD}}
\newcommand{\REG}{{\rm REG}}

\newcommand{\GCH}{{\rm GCH}}
\newcommand{\SCH}{{\rm SCH}}

\newcommand{\ORD}{\mathop{{\rm ORD}}}

\renewcommand{\P}{{\mathbb P}}
\newcommand{\Q}{{\mathbb Q}}

\renewcommand{\S}{\mathbb{S}}

\newcommand{\QQ}{Q}

\newcommand{\Add}{\mathop{\rm Add}}

\newcommand{\st}{\mid}
\newcommand{\restrict}{\upharpoonright}

\renewcommand{\>}{\rangle}
\newcommand{\elemsub}{\prec}
\newcommand{\elesub}{\prec}

\newcommand{\dom}{\mathop{\rm dom}}
\newcommand{\ran}{\mathop{\rm ran}}

\newcommand{\cf}{\mathop{\rm cf}}

\renewcommand{\and}{\mathop{\&}}

\newtheorem{theorem}{Theorem}
\newtheorem{lemma}{Lemma}
\newtheorem{corollary}{Corollary}

\theoremstyle{definition}
\newtheorem{question}{Question}
\newtheorem{remark}{Remark}

\newtheorem{definition}{Definition}

\date{\today}

\begin{document}

\title{Easton functions and supercompactness}

\author[Brent Cody]{Brent Cody}
\address[Brent Cody]{ 
Virginia Commonwealth University,
Department of Mathematics and Applied Mathematics,
1015 Floyd Avenue, Richmond, Virginia 23284, United States.} 
\email[B. ~Cody]{bmcody@vcu.edu} 

\author[Sy-David Friedman]{Sy-David Friedman}
\address[Sy-David Friedman]{ 
The University of Vienna,
Kurt G\"odel Research Center for Mathematical Logic,
W\"ahringer Strasse 25, A-1090 Wien, Austria.} 
\email[S. ~Friedman]{sdf@logic.univie.ac.at} 
\urladdr{}

\author[Radek Honzik]{Radek Honzik}
\address[Radek Honzik]{ 
Kurt G{\"o}del Research Center for Mathematical Logic, Vienna and Department of Logic, Charles University, Prague, Czech Republic.} 
\email[R. ~Honzik]{radek.honzik@ff.cuni.cz} 
\urladdr{}

\maketitle


\begin{abstract}

Suppose $\kappa$ is $\lambda$-supercompact witnessed by an elementary embedding $j:V\to M$ with critical point $\kappa$, and further suppose that $F$ is a function from the class of regular cardinals to the class of cardinals satisfying the requirements of Easton's theorem: (1) $\forall\alpha$ $\alpha<\cf(F(\alpha))$ and (2) $\alpha<\beta$ $\implies$ $F(\alpha)\leq F(\beta)$. In this article we address the question: assuming $\GCH$, what additional assumptions are necessary on $j$ and $F$ if one wants to be able to force the continuum function to agree with $F$ globally, while preserving the $\lambda$-supercompactness of $\kappa$?

We show that, assuming $\GCH$, if $F$ is any function as above, and in addition for some regular cardinal $\lambda>\kappa$ there is an elementary embedding $j:V\to M$ with critical point $\kappa$ such that $\kappa$ is closed under $F$, the model $M$ is closed under $\lambda$-sequences, $H(F(\lambda))\subseteq M$, and for each regular cardinal $\gamma\leq \lambda$ one has $(|j(F)(\gamma)|=F(\gamma))^V$, \underline{then} there is a cardinal-preserving forcing extension in which $2^\delta=F(\delta)$ for every regular cardinal $\delta$ and $\kappa$ remains $\lambda$-supercompact. This answers a question of \cite{CodyMagidor}.

\end{abstract}

\section{Introduction}\label{sectionintroduction}

In this article we address the following question, which is posed in \cite{CodyMagidor}. 
\begin{question}\label{mainquestion}
Given a $\lambda$-supercompact cardinal $\kappa$ and assuming $\GCH$, what behaviors of the continuum function on the regular cardinals can be forced while preserving the $\lambda$-supercompactness of $\kappa$, and from what hypotheses can such behaviors of the continuum function be obtained? 
\end{question}

Let us first consider the special case where $\kappa$ is $\kappa$-supercompact, in other words $\kappa$ is measurable. Silver proved that if $\kappa$ is $\kappa^{++}$-supercompact and $\GCH$ holds, then there is a cofinality-preserving forcing extension in which $\kappa$ remains measurable and $2^\kappa=\kappa^{++}$, but one can also obtain such a model from a much weaker hypothesis. Woodin proved that the existence of a measurable cardinal $\kappa$ such that $2^\kappa=\kappa^{++}$ is equiconsistent with the existence of an elementary embedding $j:V\to M$ with critical point $\kappa$ such that $j(\kappa)>\kappa^{++}$ and $M^\kappa\subseteq M$. The forward direction of Woodin's equiconsistency is trivial, and for the backward direction the embedding is lifted to a certain forcing extension $V[G][H][g_0]$ where $g_0$ is an ``extra forcing'' necessary for carrying out a surgical modification of a generic filter on the $M$-side (see \cite[Theorem 25.1]{Cummings:Handbook} or \cite[Theorem 36.2]{Jech:Book}). A more uniform method for proving Woodin's equiconsistency, in which no ``extra forcing'' is required, is given in \cite{FriedmanThompson:PerfectTreesAndElementaryEmbeddings}. This method involves lifting an elementary embedding through Sacks forcing on uncountable cardinals, an idea which has found many additional applications (see \cite{FriedmanMagidor:TheNumberOfNormalMeasures}, \cite{FriedmanHonzik:EastonsTheoremAndLargeCardinals}, \cite{FriedmanZdomskyy2010}, \cite{Honzik-GlobalSingularization}, \cite{DobrinenFriedman:HomogeneousIteration}, \cite{FriedmanHonzik:EastonsThmAndLCFromOptimal} and \cite{FriedmanHonzik:SupercompactnessAndFailuresOfGCH}). The uniformity of the method led to answers \cite{FriedmanHonzik:EastonsTheoremAndLargeCardinals} to Question 1 in the case that $\kappa$ is a measurable cardinal and in the case that $\kappa$ is a strong cardinal. 

In a result analagous to Woodin's equiconsistency mentioned above, the first author proved \cite{Cody:TheFailureOfGCHAtADegreeOfSupercompactness} the equiconsistency of the following three hypotheses.
\begin{itemize}
\item[(i)] There is a cardinal $\kappa$ that is $\lambda$-supercompact and $2^\kappa>\lambda^{++}$.
\item[(ii)] There is a cardinal $\kappa$ that is $\lambda$-supercompact and $2^\lambda>\lambda^{++}$.
\item[(iii)] There is an elementary embedding $j:V\to M$ with critical point $\kappa$ such that $j(\kappa)>\lambda^{++}$ and $M^\lambda\subseteq M$.
\end{itemize}
In the argument of \cite{Cody:TheFailureOfGCHAtADegreeOfSupercompactness}, a model of (ii) is obtained from a model of (iii) by lifting the embedding $j$ to a forcing extension of the form $V[G][H][g_0]$ by using Woodin's technique of surgically modifying a generic filter. However, in the final model, $\kappa$ is $\lambda$-supercompact and one has $2^\kappa=2^\lambda=\lambda^{++}$, so the final model satisfies both (i) and (ii). Furthermore, it is remarked in \cite{Cody:TheFailureOfGCHAtADegreeOfSupercompactness} that the surgery argument does not seem to yield a model with $\GCH$ on the interval $[\kappa,\lambda)$ and $2^\lambda=\lambda^{++}$, where $\kappa$ is $\lambda$-supercompact.

The second and third authors showed that the more uniform method involving Sacks forcing on uncountable cardinals can be used to address this discordance. Indeed, it is shown in \cite{FriedmanHonzik:SupercompactnessAndFailuresOfGCH} that from the hypothesis (iii) above, and assuming $\GCH$, there is a cofinality-preserving forcing extension in which $\kappa$ remains $\lambda$-supercompact, $\GCH$ holds on the interval $[\kappa,\lambda)$ and $2^\lambda=\lambda^{++}$. The following question is posed in \cite{FriedmanHonzik:SupercompactnessAndFailuresOfGCH}. Starting with a model of (iii) and $\GCH$, is there a cofinality-preserving forcing extension in which $\kappa$ is $\lambda$-supercompact and for some regular cardinal $\gamma$ with $\kappa<\gamma<\lambda$ one has $\GCH$ on $[\kappa,\gamma)$ and $2^\gamma=\lambda^{++}$? This question was recently answered in \cite{CodyMagidor} where it is shown that that Woodin's method of surgically modifying a generic filter to lift an embedding can be extended to include the case where modifications are made on ``ghost-coordinates.'' Indeed \cite{CodyMagidor} estabilshes
that if $\GCH$ holds, $F:[\kappa,\lambda]\cap\REG\to\CARD$ is any function satisfying Easton's requirements 
\begin{itemize}
\item[(E1)] $\alpha<\cf(F(\alpha))$ and 
\item[(E2)] $\alpha<\beta$ implies $F(\alpha)\leq F(\beta)$
\end{itemize}
where $\lambda>\kappa$ is a regular cardinal, and there is a $j:V\to M$ with critical point $\kappa$ such that $j(\kappa)>F(\lambda)$ and $M^\lambda\subseteq M$, then there is a cofinality-preserving forcing extension in which $\kappa$ remains $\lambda$-supercompact and $2^\gamma=F(\gamma)$ for every regular cardinal $\gamma$ with $\kappa\leq\gamma\leq\lambda$. This provides an answer to the above Question \ref{mainquestion} if we restrict our attention to controlling the continuum function only on the interval $[\kappa,\lambda]$ while preserving the $\lambda$-supercompactness of $\kappa$. 

In this article we combine the methods of \cite{FriedmanHonzik:EastonsTheoremAndLargeCardinals} and \cite{CodyMagidor} to address Question \ref{mainquestion} in the context of controlling the continuum function at all regular cardinals by proving the following theorem.

\begin{theorem}\label{theorem1}
Assume $\GCH$. Suppose $F:\REG\to\CARD$ is a function satisfying Easton's requirements (E1) and (E2), for some regular cardinal $\lambda>\kappa$ 
there is an elementary embedding $j:V\to M$ with critical point $\kappa$ such that $\kappa$ is closed under $F$, the model $M$ is closed under $\lambda$-sequences, $H(F(\lambda))\subseteq M$, and for each regular cardinal $\gamma\leq \lambda$ one has $(|j(F)(\gamma)|=F(\gamma))^V$. Then there is a cardinal-preserving forcing extension in which $2^\delta=F(\delta)$ for every regular cardinal $\delta$ and $\kappa$ remains $\lambda$-supercompact

\end{theorem}

The forcing used to prove Theorem \ref{theorem1} will be an Easton-support iteration of Easton-support products of Cohen forcing. To lift the embedding through the first $\kappa$-stages of the forcing, we will use the technique of twisting a generic using an automorphism in order to obtain a generic for the $M$-side (see \cite{FriedmanHonzik:EastonsTheoremAndLargeCardinals}). In order to lift the embedding through through a later portion of the iteration we will use the technique of surgically modifying a generic filter on ghost-coordinates (see \cite{CodyMagidor}), which will require us to use an ``extra forcing'' over $V$. We will prove a lemma which establishes not only that the extra forcing preserves cardinals, but it also does not disturb the continuum function (see Lemma \ref{lemmadonotdisturb} below). Note that the later was not necessary in \cite{CodyMagidor}.

Regarding the hypothesis of Theorem \ref{theorem1}, notice that if $j:V\to M$ witnesses the $\lambda$-supercompactness of $\kappa$ then it follows that for $\gamma\leq\lambda$ we have $2^\gamma\leq(2^\gamma)^M<j(\kappa)$ and futhermore, in $V$, the cardinality of $(2^\gamma)^M$ is equal to $2^\gamma$. Thus, if one desires to lift an embedding $j:V\to M$ to a forcing extension in which the continuum function agrees with some $F$ as in the statement of Theorem \ref{theorem1}, then one must require that $(|j(F)(\gamma)|=F(\gamma))^V$.

\section{Preliminaries}

We assume familiarity with Easton's theorem \cite{Easton:PowersOfRegularCardinals} as well as with lifting large cardinal embeddings through forcing, see \cite{Cummings:Handbook}.

In the proof of Theorem \ref{theorem1} we will use the following forcing notion. Suppose $F$ is a function from the regular cardinals to the cardinals satisfying the requirements (E1) and (E2) of Easton's theorem and that $\kappa<\lambda$ are regular cardinals. We will let $\Q_{[\kappa,\lambda]}$ denote the Easton-support product of Cohen forcing that will ensure that, assuming $\GCH$ in the ground model, the continuum function agrees with $F$ on $[\kappa,\lambda]\cap\REG$ in the forcing extension. We can regard conditions $p\in \Q_{[\kappa,\lambda]}$ as functions satisfying the following.

\begin{itemize}
\item Every element in $\dom(p)$ is of the form $(\gamma,\alpha,\beta)$ where $\gamma\in [\kappa,\lambda]$ is a regular cardinal, $\alpha<\gamma$, and $\beta<F(\gamma)$.
\item (Easton support) For each regular cardinal $\gamma\in[\kappa,\lambda]$ we have 
$$|\{(\delta,\alpha,\beta)\in\dom(p)\mid\delta\leq\gamma\}|<\gamma.$$
\item $\ran(p)\subseteq\{0,1\}$.
\end{itemize}

\begin{lemma}[\cite{Easton:PowersOfRegularCardinals}]\label{lemmaeaston}
Assuming $\GCH$, forcing with the poset $\Q_{[\kappa,\lambda]}$ preserves all cofinalities and achieves $2^\gamma=F(\gamma)$ for every regular cardinal $\gamma\in[\kappa,\lambda]$ while preserving $\GCH$ otherwise.
\end{lemma}

\begin{remark}\label{remarkextender}
Suppose $F$ and $j$ are as in the hypothesis of Theorem \ref{theorem1}. Let us briefly show that one can assume, without loss of generality, that $M$ is of the form
\[M=\{j(f)(j"\lambda,\alpha)\st f:P_\kappa\lambda\times\kappa\to V \ \and\  \alpha<F(\lambda) \ \and\  f\in V\}.\]

Let $j:V\to M$ be as in the statement of Theorem \ref{theorem1}. We will show that $j$ can be factored through an embedding $j_0:V\to M_0$ having all the desired properties. Let $X_0=\{j(f)(j"\lambda,\alpha)\st f:P_\kappa\lambda\times\kappa\to V \ \and\  \alpha< F(\lambda) \ \and\  f\in V\}$ 
and 
$X_1=\{j(f)(j"\lambda, a)\st f:P_\kappa\lambda\times H(\kappa)\to V \ \and\  a\in H(F(\lambda))\ \and\  f\in V\}$.
Now let $\pi_0:X\to M_0$ and $\pi_1:X_1\to M_1$ be the Mostowski collapses of $X_0$ and $X_1$ respectively. Define $j_0:=\pi_0^{-1}\circ j:V\to M_0$ and $j_1:=\pi_1^{-1}\circ j:V\to M_1$. It follows that $j_0:V\to M_0$ has critical point $\kappa$, $M_0^\lambda\subseteq M_0$, $j_0(\kappa)>F(\lambda)$ and $M_0$ has the desired form 
\[M_0=\{j_0(f)(j_0"\lambda,\alpha)\st f:P_\kappa\lambda\times\kappa\to V\ \and\  \alpha<F(\lambda)\ \and\  f\in V\}.\]
It remains to show that $H(F(\lambda))\subseteq M_0$. It is easy to see that $H(F(\lambda))\subseteq M_1$ using the fact that 
\[M_1=\{j_1(f)(j_1"\lambda,a)\st f:P_\kappa\lambda\times H(\kappa)\to V\ \and\  a\in H(F(\lambda))\ \and\  f\in V\}.\]
Since the map $i:M_0\to M_1$ defined by $i(j_0(f)(j_0"\lambda,\alpha)):=j_1(f)(j_1"\lambda,\alpha)$ is an elementary embedding with critical point greater than $F(\lambda)$, and since $i$ is the identity on $F(\lambda)$, it follows that $i$ is surjective, and thus $H(F(\lambda))\subseteq M_0=M_1$. To see that $i$ is surjective onto $H(F(\lambda))$ (and thus onto $M_1$) one uses the fact that each $x\in H(F(\lambda))$ can be coded by a subset of some cardinal $\delta<F(\lambda)$.

\end{remark}

\section{Proof of Theorem \ref{theorem1}}

\begin{proof}[Proof of Theorem 1]

Our final model will be a forcing extension of $V$ by an $\ORD$-length forcing iteration $\P$, which will be broken up as $\P\cong\P^1*\dot{\S}*\dot{\P}^2$. The first factor $\P^1$, will be an iteration forcing the continuum function to agree with $F$ at every regular cardinal less than or equal to $F(\lambda)$. The second factor $\S$ will be an ``extra forcing'' that will be necessary to cary out the surgery argument to lift the embedding $j$ through $\P^1$. We will argue that the extra forcing $\S$ is mild in $V^{\P^1}$ in the sense that it preserves all cofinalities and preserves the continuum function. The last factor $\P^2\in V^{\P^1*\dot\S}$ will be a ${\leq}F(\lambda)$-closed, $\ORD$ length Easton-support product of Cohen forcing, which will force the continuum function to agree with $F$ at all regular cardinals greater than or equal to $F(\lambda)^+$.

For an ordinal $\alpha$ let $\bar{\alpha}$ denote the least closure point of $F$ greater than $\alpha$. For a regular cardinal $\gamma$, the notation $\Add(\gamma,F(\gamma))$ denotes the forcing poset for adding $F(\gamma)$ Cohen subsets to $\gamma$.

Let $\lambda_0$ be the greatest closure point of $F$ which is less or equal to $\lambda$. We now recursively define an Easton-support forcing iteration $\P_{\lambda_0+1}=\langle (\P_{\eta},\dot{\Q}_{\eta}) : \eta\leq\lambda_0 \rangle$ as follows.
\begin{enumerate}
\item If $\eta<\lambda_0$ is a closure point of $F$, then $\dot{\Q}_{\eta}$ is a $\P_{\eta}$-name for the Easton support product \[\Q_{[\eta,\bar{\eta})}=\prod_{\gamma\in[\eta,\bar{\eta})\cap\REG}\Add(\gamma,F(\gamma))\] 
as defined in $V^{\P_{\eta}}$ and $\P_{\eta+1}=\P_{\eta}*\dot{\Q}_{\eta}$.
\item If $\eta=\lambda_0$, then $\dot{\Q}_\eta$ is a $\P_{\lambda_0}$-name for \[\Q_{[\lambda_0,F(\lambda)]}=\prod_{\gamma\in[\lambda_0,F(\lambda)]\cap\REG}\Add(\gamma,F(\gamma))\] 
as defined in $V^{\P_{\lambda_0}}$ and $\P_{\lambda_0+1}=\P_{\lambda_0}*\dot{\Q}_{\lambda_0}$.
\item Otherwise, if $\eta<\lambda_0$ is not a closure point of $F$, then $\dot{\Q}_\eta$ is a $\P_\eta$-name for trivial forcing and $\P_{\eta+1}=\P_\eta*\dot{\Q}_\eta$.
\end{enumerate}

Let $G_{\lambda_0+1}$ be generic for $\P_{\lambda_0+1}$ over $V$.

\begin{remark}[Notation] We will adopt the notation and conventions used in \cite{FriedmanHonzik:EastonsTheoremAndLargeCardinals}. We will use $\prod_{[\eta,\bar{\eta})}\QQ_\gamma$ to denote $\Q_{[\eta,\bar{\eta})}$ where $Q_\gamma:=\Add(\gamma,F(\gamma))$ denotes an individual factor of the product, and similarly $g_{[\eta,\bar{\eta})}$ denotes the corresponding generic filter. It will be understood that, for example, $g_{[\eta,\bar{\eta})}$ is a product over just the regular cardinals in the interval $[\eta,\bar{\eta})$ of the relevant generic filters. In particular, if $\eta$ is a singular cardinal then there is no forcing over $\eta$ in the product $g_{[\eta,\bar{\eta})}\subseteq\prod_{[\eta,\bar{\eta})}\QQ_\gamma$.

\end{remark}

\subsection{Lifting the embedding through $\P_\kappa$ by twisting a generic using an automorphism.}
By Remark \ref{remarkextender} we can assume that $j:V\to M$ is an embedding as in the statement of Theorem \ref{theorem1} such that 
\[M=\{j(f)(j"\lambda,\alpha)\st f:P_\kappa\lambda\times \kappa\to V \ \and\  \alpha<F(\lambda)\ \and\  f\in V\}.\]

First we will lift $j$ through $G_\kappa\subseteq\P_\kappa$ by finding a filter for $j(\P_\kappa)$ that is generic over $M$. We will need the following definitions of various cardinals relating to $F$ and $\lambda$.

\begin{definition}\label{cardinalsdefinition}
The first three definitions will be needed because the forcings $\P_{\lambda_0+1}$ and $j(\P_{\lambda_0+1})$ are iterations of products over intervals determined by closure points of $F$ and $j(F)$ respectively, and these three cardinals are important such closure points.
\begin{itemize}
\item $\lambda_0:=\textrm{``the greatest closure point of $F$ that is at most $\lambda$''}$
\item $\lambda_1:=\textrm{``the least closure point of $F$ greater than $\lambda_0$''}$
\item $\lambda_1^M:=\textrm{``the least closure point of $j(F)$ greater than $\lambda_0$''}$
\end{itemize}
The way one builds a generic for the forcing $\Add(\gamma,F(\gamma))$ depends, of course, on the size of $F(\gamma)$ and the regular cardinals $\gamma_0$ and $\gamma_1$ defined below are important transition points in the size of $F(\gamma)$.
\begin{itemize}
\item $\gamma_0:=\textrm{``the least regular cardinal less than or equal to $\lambda$ such that $F(\gamma_0)=F(\lambda)$''}$
\item $\gamma_1:=\textrm{``the least regular cardinal such that $F(\gamma_1)>F(\lambda)$''}$
\end{itemize}
\end{definition}

We have $\kappa\leq \lambda_0\leq\gamma_0\leq \lambda <\gamma_1\leq F(\lambda)=F(\gamma_0)\leq j(F)(\gamma_0)<\lambda_1^M <j(\kappa)<F(\lambda)^+<\lambda_1$. Furthermore, if $\gamma\in [\kappa,\gamma_0)$ is a regular cardinal we have $|j(F)(\gamma)|^V=F(\gamma)$ and since $M$ and $V$ have the same cardinals $\leq F(\lambda)$, it follows that $j(F)(\gamma)=F(\gamma)$. In other words, $F$ and $j(F)$ agree on $[\kappa,\gamma_0)\cap\REG$. This implies that we may let $G^M_{[\kappa,\lambda_0)}=G_{[\kappa,\lambda_0)}$ and $g^M_{[\lambda_0,\gamma_0)}=g_{[\lambda_0,\gamma_0)}$.
 Note that $F$ and $j(F)$ may disagree at $\gamma_0$ because $M$ has cardinals strictly between $F(\gamma_0)=F(\lambda)$ and $(F(\lambda)^+)^V$. 

Suppose $\gamma\in[\gamma_0,F(\lambda)]$ is a regular cardinal. Since $j(\kappa)$ is a closure point of $j(F)$ we have $F(\lambda)\leq j(F)(\gamma)<j(\kappa)$, and since $|j(\kappa)|^V\leq F(\lambda)$ it follows that $|j(F)(\gamma)|^V=F(\lambda)$. Let us define a forcing $\prod_{[\gamma_0,F(\lambda)]}\QQ^+_\gamma$ in $V[G_{\lambda_0}]$ that will be used to obtain a generic for $\Q^M_{[\gamma_0, F(\lambda)]}$ over $M[G_{\lambda_0}]$. Working in $V[G_{\lambda_0}]$, for regular $\gamma\in [\gamma_0,\gamma_1)$ let $\QQ_\gamma^*=\Add(\gamma,j(F)(\gamma))$ and notice that $\QQ^*_\gamma$ is isomorphic to $\Add(\gamma,F(\gamma))$ since $|j(F)(\gamma)|^V=|j(\kappa)|^V=F(\lambda)=F(\gamma)$. For regular $\gamma\in[\gamma_1,F(\lambda)]$, let $\QQ_\gamma^{**}=\Add(\gamma,j(F)(\gamma))$ and notice that $\QQ_\gamma^{**}$ is a truncation of $\Add(\gamma,F(\gamma))$ because for such $\gamma$ one has $j(F)(\gamma)<j(\kappa)<F(\lambda)^+\leq F(\gamma)$. Now define 
\begin{align}
\prod\nolimits_{[\gamma_0,F(\lambda)]}\QQ_\gamma^+:=\prod\nolimits_{[\gamma_0,\gamma_1)}\QQ_\gamma^* \times\prod\nolimits_{[\gamma_1,F(\lambda)]}\QQ_\gamma^{**}.\label{eqnplus}
\end{align}
It is easy to see that $\prod\nolimits_{[\gamma_0,F(\lambda)]}\QQ_\gamma^+$ completely embeds into $\prod_{[\gamma_0,F(\lambda)]}\QQ_\gamma$, and hence there is a filter $g^+_{[\gamma_0,F(\lambda)]}\in V[G_{\lambda_0}*(g_{[\lambda_0,\gamma_0)}\times g_{[\gamma_0,F(\lambda)]})]$ generic over $V[G_{\lambda_0}*g_{[\lambda_0,\gamma_0)}]$ for $\prod\nolimits_{[\gamma_0,F(\lambda)]}\QQ_\gamma^+$.

The lifting of $j$ through $G_\kappa$ will be broken up into two cases, depending on the regularity or singularity of $F(\lambda)$. If $F(\lambda)$ is regular, the proof is substantially simpler because it almost directly follows from the assumption $H(F(\lambda))\subseteq M$ (see Lemma \ref{lemmaregular} below). If $F(\lambda)$ is singular, there are two cases to distinguish depending on whether the $V$-cofinality of $F(\lambda)$ is $\lambda^+$ or not; in both cases the assumption of $H(F(\lambda))\subseteq M$ is again essential, but an additional argument is required. Assuming $F(\lambda)$ is singular, the case in which $\cf(F(\lambda))^V=\lambda^+$ is easier to handle than the case where $\cf(F(\lambda))>\lambda^+$. The later case requires an induction along a matrix of coordinates (see Lemma \ref{lemmasingular}). To avoid long repetitions of the relevant proofs in \cite{FriedmanHonzik:EastonsTheoremAndLargeCardinals}, we include only outlines of the proofs of Lemma \ref{lemmaregular} and Lemma \ref{lemmasingular}, with detailed references to \cite{FriedmanHonzik:EastonsTheoremAndLargeCardinals} where appropriate (the proofs in \cite{FriedmanHonzik:EastonsTheoremAndLargeCardinals} apply almost verbatim here when one identifies $\kappa$ with $\lambda$).

\begin{lemma}\label{lemmaregular}
Assume $F(\lambda)$ is regular. There is in $V[G_{\lambda_0}*(g_{[\lambda_0,\gamma_0)}\times g_{[\gamma_0,F(\lambda)]})]$ an $M[G_{\lambda_0}*g_{[\lambda_0,\gamma_0)}]$-generic for $\prod^M_{[\gamma_0,\lambda_1^M)}\QQ^M_\gamma$, which we will denote as $g^M_{[\gamma_0,\lambda_1^M)}$. Furthermore, we can take $g^M_{[\gamma_0,\lambda]}$ to agree with $g^+_{[\gamma_0,\lambda]}$, that is, $g^M_{[\gamma_0,\lambda_1^M)}=g^+_{[\gamma_0,\lambda]}\times g^M_{(\lambda,\lambda_1^M)}$.
\end{lemma}

\begin{proof}
Since $F(\lambda)$ is regular in $V$ and hence also in $M$, it follows that $\prod^M_{[\gamma_0,F(\lambda)]}\QQ_\gamma^M$ is $(F(\lambda)^+)^M$-c.c. in $M[G_{\lambda_0}*g_{[\lambda_0,\gamma_0)}]$. Furthermore, $\prod^M_{(F(\lambda),\lambda_1^M)}\QQ_\gamma^M$ is $(F(\lambda)^+)^M$-closed in $M[G_{\lambda_0}*g_{[\lambda_0,\gamma_0)}]$. It follows by Easton's Lemma that generic filters for these forcings are mutually generic and therefore it suffices to obtain generic filters for them separately.


As in \cite[Lemma 3.9]{FriedmanHonzik:EastonsTheoremAndLargeCardinals}, one may check that $g^M_{[\gamma_0,F(\lambda)]}:=g^+_{[\gamma_0,F(\lambda)]}\cap\prod^M_{[\gamma_0,F(\lambda)]}\QQ_\gamma^M$ is $M[G_{\lambda_0}*g_{[\lambda_0,\gamma_0)}]$-generic and one can build an $M[G_{\lambda_0}*g_{[\lambda_0,\gamma_0)}]$-generic filter $g^M_{(F(\lambda),\lambda_1^M)}$ for $\prod^M_{(F(\lambda),\lambda_1^M)}\QQ^M_\gamma$ in $V[G_{\lambda_0}*g_{[\lambda_0,\gamma_0)}]$. 

Now we define $g^M_{[\gamma_0,\lambda_1^M)}:=g^M_{[\gamma_0,F(\lambda)]}\times g^M_{(F(\lambda),\lambda_1^M)}$ and it remains to show that $g^M_{[\gamma_0,\lambda_1^M)}=g^+_{[\gamma_0,\lambda]}\times g^M_{(\lambda,\lambda_1^M)}$. Since $M[G_{\lambda_0}]$ is closed under $\lambda$-sequences in $V[G_{\lambda_0}]$ we have $\prod^M_{[\gamma_0,\lambda]}\QQ^M_\gamma=\prod_{[\gamma_0,\lambda]}\QQ_\gamma^*$. Now use (\ref{eqnplus}) to obtain the desired conclusion.
\end{proof}

\begin{lemma}\label{lemmasingular}
Assume $F(\lambda)$ is singular. There is in $V[G_{\lambda_0}*(g_{[\lambda_0,\gamma_0)}\times g_{[\gamma_0,F(\lambda))})]$ an $M[G_{\lambda_0}*g_{[\lambda_0,\gamma_0)}]$-generic for $\prod^M_{[\gamma_0,\lambda_1^M)}\QQ^M_\gamma$, which we will denote as $g^M_{[\gamma_0,\lambda_1^M)}$. Furthermore, we can take $g^M_{[\gamma_0,\lambda_1^M)}$ to be of the form $\sigma[g^+_{[\gamma_0,\lambda]}]\times g^M_{(\lambda,\lambda_1^M)}$ where $\sigma$ is an automorphism of $\prod_{[\gamma_0,\lambda]}\QQ^+_\gamma$ in $V[G_{\lambda_0}]$ and $g^M_{(\lambda,\lambda_1^M)}$ is $M[G_{\lambda_0}*g_{[\lambda_0,\gamma_0)}]$-generic for $\prod^M_{(\lambda,\lambda_1^M)}\QQ^M_\gamma$.
\end{lemma}

\begin{proof} 
{\bf Case I:} Suppose $F(\lambda)$ is singular in $V$ with $\cf(F(\lambda))^V=\lambda^+$ ($F(\lambda)$ can be singular or regular in $M$). As in \cite[Sublemma 3.12]{FriedmanHonzik:EastonsTheoremAndLargeCardinals}, we can find a condition $p_\infty\in\prod_{[\gamma_0,\lambda_1^M)}\QQ^+_\gamma$ (which may only exist in $V[G_{\lambda_0}*g_{[\lambda_0,\gamma_0)}]$) such that if $h$ is generic for $\prod_{[\gamma_0,F(\lambda))}\QQ_\gamma^+$ with $p_\infty\restrict[\gamma_0,F(\lambda))\in h$ and $h'=\{p_\infty\restrict[F(\lambda),\lambda_1^M)\}\cup\{q\in\prod^M_{[F(\lambda),\lambda_1^M)}\QQ^M_\gamma\st p_\infty\restrict[F(\lambda),\lambda_1^M)\leq q\}$, then $(h\times h')\cap M[G_{\lambda_0}*g_{[\lambda_0,\gamma_0)}]$ is $M[G_{\lambda_0}*g_{[\lambda_0,\gamma_0)}]$-generic for $\prod^M_{[\gamma_0,\lambda_1^M)}\QQ_\gamma^M$.

We define $g^M_{[\gamma_0,\lambda_1^M)}$ as follows. A homogeneity argument can be used to find an automorphism $\sigma$ of $\prod_{[\gamma_0,F(\lambda))}\QQ_\gamma^+$ such that $p_\infty\restrict[\gamma_0,F(\lambda))\in \sigma[g^+_{[\gamma_0,F(\lambda))}]$. We obtain the desired generic by letting 
\begin{align}
g^M_{[\gamma_0,\lambda_1^M)}:=\left(\sigma[g^+_{[\gamma_0,F(\lambda))}]\times h'\right)\cap M[G_{\lambda_0}*g_{[\lambda_0,\gamma_0)}].\label{eqndefofgen}
\end{align}
Since $M[G_{\lambda_0}]$ is closed under $\lambda$-sequences in $V[G_{\lambda_0}]$ we have $\prod^M_{[\gamma_0,\lambda]}\QQ^M_\gamma=\prod_{[\gamma_0,\lambda]}\QQ_\gamma^*$. Now using (\ref{eqnplus}) and the definition (\ref{eqndefofgen}), we obtain $g^M_{[\gamma_0,\lambda]}=\sigma[g^+_{[\gamma_0,\lambda]}]$.

{\bf Case II:} Suppose $F(\lambda)$ is singular in $V$ and $\cf(F(\lambda))^V\neq \lambda^+$ ($F(\lambda)$ can be singular or regular in $M$). If $F(\lambda)$ is regular in $M$ then, as in \cite[Sublemma 3.13]{FriedmanHonzik:EastonsTheoremAndLargeCardinals} we can use a ``matrix of confitions'' argument to find a $p_\infty$ as above. As in Case I, we have $g^M_{[\gamma_0,\lambda_1^M)}:=\sigma[g^+_{[\gamma_0,F(\lambda))}]\times h'$ is $M[G_{\lambda_0}*g_{[\lambda_0,\gamma_0)}]$-generic for $\prod^M_{[\gamma_0,\lambda_1^M)}\QQ_\gamma^M$ where $h'$ is some $M[G_{\lambda_0}*g_{[\lambda_0,\gamma_0)}]$-generic filter for $\prod^M_{[F(\kappa),\lambda_1^M)}\QQ^M_\gamma$. As in Case I we get $g^M_{[\gamma_0,\lambda]}=\sigma[g^+_{[\gamma_0,\lambda]}]$.

If $F(\lambda)$ is singular in $M$ then an easier argument will suffice (see \cite[Case (2), page 205]{FriedmanHonzik:EastonsTheoremAndLargeCardinals}).
\end{proof}

By Lemmas \ref{lemmaregular} and \ref{lemmasingular} above, if $F(\lambda)$ is regular or singular in $V$, there is an $M[G_{\lambda_0}*g_{[\lambda_0,\gamma_0)}]$-generic filter $g^M_{[\gamma_0,\lambda_1^M)}$ for $\prod^M_{[\gamma_0,\lambda_1^M)}\QQ^M_\gamma$ in \[V[G_{\lambda_0}*g_{[\lambda_0,F(\lambda)]}] = V[G_{\lambda_0}*(g_{[\lambda_0,\gamma_0)}\times g_{[\gamma_0,F(\lambda)]})].\] 
Define $g^M_{[\lambda_0,\lambda_1^M)}:=g_{[\lambda_0,\gamma_0)}\times g^M_{[\gamma_0,\lambda_1^M)}$. We will now use the fact that, depending on whether $F(\lambda)$ is regular or singular, $g^M_{[\gamma_0,\lambda_1^M)}$ agrees with either $g^+_{[\gamma_0,\lambda]}$ or an automorphic image of $g^+_{[\gamma_0,\lambda]}$ to establish the following.

\begin{lemma}\label{lemmaclosureaftertwisting}
$M[G_{\lambda_0}*g^M_{[\lambda_0,\lambda_1^M)}]$ is closed under $\lambda$-sequences in $V[G_{\lambda_0}*g_{[\lambda_0,F(\lambda)]}]$.
\end{lemma}

\begin{proof}
It will suffice to argue that if $X$ is a $\lambda$-sequence of ordinals in $V[G_{\lambda_0}*(g_{[\lambda_0,\gamma_0)}\times g_{[\gamma_0,F(\lambda)]})]$ then $X$ is in $M[G_{\lambda_0}*(g_{[\lambda_0,\gamma_0)}\times g^M_{[\gamma_0,\lambda_1^M)})]$. Since $\prod_{(\lambda, F(\lambda)]}\QQ_\gamma$ is $\leq\lambda$-distributive in $V[G_{\lambda_0}*g_{[\lambda_0,\lambda]}]$ we have 
$X\in V[G_{\lambda_0}*g_{[\lambda_0,\lambda]}]$. Furthermore, since $\lambda<\gamma_1$ it follows from (\ref{eqnplus}) that 
\[X \in V[G_{\lambda_0}*g_{[\lambda_0,\lambda]}]=V[G_{\lambda_0}*(g_{[\lambda_0,\gamma_0)}\times g^+_{[\gamma_0,\lambda]})].\]

Since $M[G_{\lambda_0}]$ is closed under $\lambda$-sequences in $V[G_{\lambda_0}]$, it follows that 
\begin{align}
\prod\nolimits^M_{[\lambda_0,\lambda]}\QQ^M_\gamma =  \prod\nolimits_{[\lambda_0,\gamma_0)}\QQ_\gamma\times\prod\nolimits_{[\gamma_0,\lambda]}\QQ^*_\gamma. \label{isoposets}
\end{align}


First let us assume that $F(\lambda)$ is regular so that, by Lemma \ref{lemmaregular}, we have $g^M_{[\lambda_0,\lambda]}=g_{[\lambda_0,\gamma_0)}\times g^+_{[\gamma_0,\lambda]}$.  As the forcing in (\ref{isoposets}) is isomorphic to $\prod_{[\lambda_0,\lambda]}\QQ_\gamma$ in $V[G_{\lambda_0}]$, we see that it is $\lambda^+$-c.c. in $V[G_{\lambda_0}]$, and therefore the model \[M[G_{\lambda_0}*(g_{[\lambda_0,\gamma_0)}\times g^+_{[\gamma_0,\lambda]})]=M[G_{\lambda_0}*g^M_{[\lambda_0,\lambda]}]\]
is closed under $\lambda$-sequences in $V[G_{\lambda_0}*g_{[\lambda_0,\lambda]}]$. Thus $X\in M[G_{\lambda_0}*g^M_{[\lambda_0,\lambda]}]\subseteq M[G_{\lambda_0}*g^M_{[\lambda_0,\lambda_1^M)}]$.

Now let us assume $F(\lambda)$ is singular. By Lemma \ref{lemmasingular} we have $g^M_{[\lambda_0,\lambda]}=g_{[\lambda_0,\gamma_0)}\times \sigma[g^+_{[\gamma_0,\lambda]}]$ for some automorphism $\sigma$ of $\prod_{[\gamma_0,\lambda]}\QQ^+_\gamma$ in $V[G_{\lambda_0}]$. Since 
\[V[G_{\lambda_0}*(g_{[\lambda_0,\gamma_0)}\times g^+_{[\gamma_0,\lambda]})]= V[G_{\lambda_0}*(g_{[\lambda_0,\gamma_0)}\times \sigma[g^+_{[\gamma_0,\lambda]}])]\]
and since $g_{[\lambda_0,\gamma_0)}\times \sigma[g^+_{[\gamma_0,\lambda]}]$ is $V[G_{\lambda_0}]$-generic for the $\lambda^+$-c.c. forcing in (\ref{isoposets}), it follows as before that the model $M[G_{\lambda_0}*(g_{[\lambda_0,\gamma_0)}\times \sigma[g^+_{[\gamma_0,\lambda]}])]=M[G_{\lambda_0}*g^M_{[\lambda_0,\lambda]}]$ is closed under $\lambda$-sequences in $V[G_{\lambda_0}*g_{[\lambda_0,\lambda]}]$. Thus $X\in M[G_{\lambda_0}*g^M_{[\lambda_0,\lambda]}]\subseteq M[G_{\lambda_0}*g^M_{[\lambda_0,\lambda_1^M)}]$.
\end{proof}


\begin{lemma}
We can build an $M[G_{\lambda_0}*(g_{[\lambda_0,\gamma_0)}\times g^M_{[\gamma_0,\lambda_1^M)})]$-generic filter $G^M_{[\lambda_1^M,j(\kappa))}$ for $\P^M_{[\lambda_1^M,j(\kappa))}$ in $V[G_{\lambda_0}*g_{[\lambda_0,F(\lambda)]}]$.
\end{lemma}

\begin{proof} 
There are at most $\lambda^+$ functions in $V$ that represent names for dense subsets of a tail of $j(\P_\kappa)$. Thus every dense subset of $\P^M_{[\lambda_1^M,j(\kappa))}$ in $M[G_{\lambda_0}*g^M_{[\lambda_0,\lambda_1^M)}]$ has a name represented by one of these functions. We may use the fact that $\P^M_{[\lambda_1^M,j(\kappa))}$ is $\leq{F(\lambda)}$-closed in $M[G_{\lambda_0}*g^M_{[\lambda_0,\lambda_1^M)}]$ and that $M[G_{\lambda_0}*g^M_{[\lambda_0,\lambda_1^M)}]$ is closed under $\lambda$-sequences in $V[G_{\lambda_0}*g_{[\lambda_0,F(\lambda)]}]$ to build a decreasing $\lambda^+$-sequence of conditions from $\P^M_{[\lambda_1^M,j(\kappa))}$ in $V[G_{\lambda_0}*g_{[\lambda_0,F(\lambda)]}]$ meeting every dense subset of $\P^M_{[\lambda_1^M,j(\kappa))}$ in $M[G_{\lambda_0}*g^M_{[\lambda_0,\lambda_1^M)}]$. It follows that this $\lambda^+$-sequence of conditions generates the desired generic filter.
\end{proof}

Thus we may lift $j$ to
\[j:V[G_\kappa]\to M[j(G_\kappa)],\]
where $j(G_\kappa)=G_{\lambda_0}*g^M_{[\lambda_0,\lambda_1^M)}*G^M_{[\lambda_1^M,j(\kappa))}$ and $j$  is a class of $V[G_{\lambda_0}*g_{[\lambda_0,F(\lambda)]}]$. Furthermore, we have that $M[j(G_\kappa)]$ is closed under $\lambda$-sequences in $V[G_{\lambda_0}*g_{[\lambda_0,F(\lambda))}]$ and 
\[M[j(G_\kappa)]=\{j(f)(j"\lambda,\alpha)\st f:P_\kappa\lambda\times\kappa\to V\ \and\ \alpha<F(\lambda)\ \and\ f\in V[G_\kappa]\}.\]

\subsection{Outline}

Our goal is to lift $j$ through the forcing $\P_{[\kappa,\lambda_0)}*\dot\Q_{[\lambda_0,\lambda]}=\P_{[\kappa,\lambda_0)}*\prod_{[\lambda_0,\lambda]}\QQ_\gamma$. Our strategy will be to first use a master condition for lifting $j$ through $\P_{[\kappa,\lambda_0)}$ of this forcing and then to use the surgery argument of \cite{CodyMagidor} to lift $j$ through $\Q_{[\lambda_0,\lambda]}$.

\subsection{Lifting the embedding through $\P_{[\kappa,\lambda_0)}$ via a master condition argument.}

In $V[G_\kappa]$, the poset $\P_{[\kappa,\lambda_0)}$ has size no larger than $\lambda$ and thus, $j"G_{[\kappa,\lambda_0)}$ has size at most $\lambda$ in $V[G_{\lambda_0}*g_{[\lambda_0, F(\lambda))}]$. Hence $j"G_{[\kappa,\lambda_0)}\in M[j(G_\kappa)]$ and since $j(\P_{[\kappa,\lambda_0)})$ is ${<}j(\kappa)$-directed closed in $M[j(G_\kappa)]$, there is a master condition $p_{[\kappa,\lambda_0)}\in j(\P_{[\kappa,\lambda_0)})$ extending every element of $j"G_{[\kappa,\lambda_0)}$. We now build an $M[j(G_\kappa)]$-generic filter below $p_{[\kappa,\lambda_0)}$. First notice that every dense subset of $j(\P_{[\kappa,\lambda_0)})$ in $M[j(G_\kappa)]$ can be written as $j(h)(j"\lambda,\alpha)$ where $h\in V[G_\kappa]$ is a function from $P_\kappa\lambda\times\kappa$ into the collection of dense subsets of $\P_{[\kappa,\lambda_0)}$ and $\alpha<F(\lambda)$. Since in $V[G_\kappa]$ there are no more than $\lambda^+$ such functions, it follows that we can enumerate them as $\<h_\xi\st\xi<\lambda^+\>\in V[G_\kappa]$ so that every dense subset of $j(\P_{[\kappa,\lambda_0)})$ in $M[j(G_\kappa)]$ is of the form $j(h_\xi)(j"\lambda,\alpha)$ for some $\xi<\lambda^+$ and some $\alpha<F(\lambda)$. One can build a decreasing $\lambda^+$-sequence of conditions $\<p_\xi\st\xi<\lambda^+\>\in V[G_{\lambda_0}*g_{[\lambda_0,F(\lambda))}]$ below $p_{[\kappa,\lambda_0)}$, such that for every $\xi<\lambda^+$ the condition $p_\xi\in j(\P_{[\kappa,\lambda_0)})$ meets every dense subset of $j(\P_{[\kappa,\lambda_0)})$ in $M[j(G_\kappa)]$ appearing in the sequence $\<j(h_\xi)(j"\lambda,\alpha)\st\alpha<F(\lambda)\>$. Let $G^M_{[j(\kappa),j(\lambda_0))}\in V[G_{\lambda_0}*g_{[\lambda_0,F(\lambda))}]$ be the filter generated by $\<p_\xi\st\xi<\lambda^+\>$. It follows by construction that $G^M_{[j(\kappa),j(\lambda_0))}$ is $M[j(G_\kappa)]$-generic and $j"G_{[\kappa,\lambda_0)}\subseteq G^M_{[j(\kappa),j(\lambda_0))}$. Thus we may lift $j$ to 
\begin{align}
j:V[G_\kappa*G_{[\kappa,\lambda_0)}]\to M[j(G_\kappa)*j(G_{[\kappa,\lambda_0)})]\label{juptotlambda_0}
\end{align}
where $j(G_{[\kappa,\lambda_0)})=G^M_{[j(\kappa),j(\lambda_0))}$ and where $j$ is a class of $V[G_{\lambda_0}*g_{[\lambda_0,F(\lambda)]}]$. Furthermore, $M[j(G_\kappa)*j(G_{[\kappa,\lambda_0)})]$ is closed under $\lambda$-sequences in $V[G_{\lambda_0}*g_{[\lambda_0,F(\lambda)]}]$.

\subsection{Obtaining a generic for $j(\Q_{[\lambda_0,\lambda]})$ for use in surgery}

Now we will lift $j$ through the forcing $\Q_{[\lambda_0,\lambda]}$ by applying the surgery technique of \cite{CodyMagidor}. We will factor the embedding in (\ref{juptotlambda_0}) through an ultrapower embedding $j_0$, force with $j_0(\Q_{[\lambda_0,\lambda]})$ over $V[G_{\lambda_0}*g_{[\lambda_0,F(\lambda)]}]$ and then modify the generic to lift the embedding.

Let $X=\{j(h)(j"\lambda)\mid h:P_\kappa\lambda\to V[G_{\lambda_0}], h\in V[G_{\lambda_0}]\}$. Then it follows that $X\elesub M[j(G_{\lambda_0})]$. Let $k:M_0'\to M[j(G_{\lambda_0})]$ be the inverse of the Mostowski collapse $\pi:X\to M_0'$ and let $j_0:V[G_{\lambda_0}]\to M_0'$ be defined by $j_0:=k^{-1}\circ j$. It follows that $j_0$ is the ultrapower embedding by the measure $U_0:=\{X\subseteq P_\kappa\lambda\mid j"\lambda\in j(X)\}$ and we will see that $U_0\in V[G_{\lambda_0}*\pi(g^M_{[\lambda_0,F(\lambda)]})]$. Using a theorem of Laver \cite{Laver:CertainVeryLargeCardinalsAreNot}, which says that the ground model is always definable from a parameter in any set forcing extension, it follows by elementarity that $M_0'$ is of the form $M_0[j_0(G_{\lambda_0})]$, where $M_0\subseteq M_0'$ and $j_0(G_{\lambda_0})\subseteq j_0(\P_{\lambda_0})\in M_0'$ is $M_0$-generic.

\begin{remark}
Since $j"\lambda\in X$ it follows that $X$ is closed under $\lambda$-sequences in $V[G_{\lambda_0}*g_{[\lambda_0,F(\lambda)]}]$. Thus $\lambda^+\subseteq X$ and hence the transitive collapse $\pi$ is the identity on $[0,\lambda^+)$. In fact $\lambda^+$ also belongs to $X$ so the critical point of $k$ is greater than $\lambda^+$.
\end{remark}

In Lemma \ref{lemmadistributive} and Lemma \ref{lemmachaincondition} below, we show that the forcing $j_0(\Q_{[\lambda_0,\lambda]})$ behaves well in the model $V[G_{\lambda_0}*g_{[\lambda_0,F(\lambda)]}]$, in the sense that it is highly distributive and has a good chain condition. Then it easily follows that forcing with $j_0(\Q_{[\lambda_0,\lambda]})$ over $V[G_{\lambda_0}*g_{[\lambda_0,F(\lambda)]}]$ preserves cardinals, and since $\SCH$ holds in $V[G_{\lambda_0}*g_{[\lambda_0,F(\lambda)]}]$, this forcing does not disturb the continuum function (see Lemma \ref{lemmadonotdisturb}).


\begin{lemma}\label{lemmadistributive}
$j_0(\Q_{[\lambda_0,\lambda]})$ is ${\leq}\lambda$-distributive in $V[G_{\lambda_0}*g_{[\lambda_0,F(\lambda)]}]$.
\end{lemma}

\begin{proof}
Define $S := j_0({\Q}_{[\lambda_0,\lambda]})$ and  $g^{M_0}_{[\lambda_0,\lambda]} := \pi(g^M_{[\lambda_0,\lambda]})$. It follows that $g^{M_0}_{[\lambda_0,\lambda]}$ is generic over $M_0[G_{\lambda_0}]$ for  $T := \pi(\prod^M_{[\lambda_0,\lambda]}\QQ^+_\gamma)$. Notice that $T$ is a ``truncated'' version of $\prod^M_{[\lambda_0,\lambda]}\QQ^+_\gamma$ because $\pi$ is the identity on $[0,\lambda]$; moreover, $g^{M_0}_{[\lambda_0,\lambda]}$ is generic for $T$ over $V[G_{\lambda_0}]$ and $T$ is $\lambda^+$-c.c. over $V[G_{\lambda_0}]$.

 We prove the lemma in two steps: (i) Firstly, we show that $M_0[j_0(G_{\lambda_0})]$ is closed under $\lambda$-sequences in $V^* := V[G_{\lambda_0}][g^{M_0}_{[\lambda_0,\lambda]} \times g_{(\lambda,F(\lambda)]}]$; this will imply that $S$ is $\le \lambda$-closed in $V^*$. (ii) Secondly, we show that $S$ remains $\le \lambda$-distributive in $V[G_{\lambda_0}*g_{[\lambda_0,F(\lambda)]}]$, which can be written -- as we will argue -- as $V^*[g_{[\lambda_0,\lambda]}]$.

As for (i), notice that $j_0|V: V \to M$ is elementary, and $M$ is closed under $\lambda$-sequences in $V$. The generic $G_{\lambda_0} * g^{M_0}_{[\lambda_0,\lambda]}$ is added by a $\lambda^+$-c.c. forcing over $V$, and hence $M_0[G_{\lambda_0}][g^{M_0}_{[\lambda_0,\lambda]}]$ is still closed under $\lambda$-sequences in $V[G_{\lambda_0}][g^{M_0}_{[\lambda_0,\lambda]}]$. Finally, the forcing adding $g_{(\lambda,F(\lambda)]}$ is, by the Easton's lemma, $\le \lambda$-distributive over $V[G_{\lambda_0}][g^{M_0}_{[\lambda_0,\lambda]}]$ (and therefore does not add new $\lambda$-sequences); now (i) follows because $M_0[j_0(G_{\lambda_0})]$ is included in $V^*$.

As for (ii), notice that $\prod_{[\lambda_0,\lambda]}\QQ_\gamma$ (with the associated generic $g_{[\lambda_0,\lambda]}$), is isomorphic in $V[G_{\lambda_0}]$ to $T \times \prod_{[\lambda_0,\lambda]}\QQ_\gamma$. Now (ii), and hence the lemma, follows by another application of the Easton's lemma, using the $\lambda^+$-cc of $\prod_{[\lambda_0,\lambda]}\QQ_\gamma$.
\end{proof}

\begin{lemma}\label{lemmachaincondition}
$j_0(\Q_{[\lambda_0,\lambda]})$ is $\lambda^{++}$-c.c. in $V[G_{\lambda_0}*g_{[\lambda_0,F(\lambda)]}]$.
\end{lemma}

\begin{proof}
Notice that each condition $p\in j_0(\Q_{[\lambda_0,\lambda]})$ can be written as $j_0(h_p)(j"\lambda)$ for some function $h_p:P_\kappa\lambda\to \Q_{[\lambda_0,\lambda]}$ in $V[G_{\lambda_0}]$. Thus, each condition $p\in j_0(\Q_{[\lambda_0,\lambda]})$ leads to a function $\bar{h}_p:\lambda\to \Q_{[\lambda_0,\lambda]}$ in $V[G_{\lambda_0}]$, which is a condition in the full-support product of $\lambda$ copies of $\Q_{[\lambda_0,\lambda]}$ taken in $V[G_{\lambda_0}]$, denoted by $\bar{\Q}=(\Q_{[\lambda_0,\lambda]})^\lambda$. 

Let us argue that $\bar{\Q}$ is $\lambda^{++}$-c.c. in $V[G_{\lambda_0}*g_{[\lambda^+,F(\lambda)]}]$. We define the domain of a condition $p=\<p_\xi\st\xi<\lambda\>\in\bar{\Q}$ to be the disjoint union of the domains of its coordinates: ${\rm domain}(p):=\bigsqcup_{\xi<\lambda}\dom(p_\xi)$. It follows that each $p\in \bar{\Q}$, being the union of $\lambda$ sets, each of size less than $\lambda$, has domain of size at most $\lambda$. Suppose $A$ is an antichain of $\bar{\Q}$ in $V[G_{\lambda_0}*g_{[\lambda^+,F(\lambda)]}]$ of size $\lambda^{++}$. If there are $\lambda^{++}$ conditions in $A$ that have a common domain, say $d$, then we immediately get a contradiction because, in $V[G_{\lambda_0}*g_{[\lambda^+,F(\lambda)]}]$, there are at most $2^\lambda=\lambda^+$ functions in $2^d$. Otherwise, the set ${\rm domain}(A)=\{{\rm domain}(p)\st p\in A\}$ has size $\lambda^{++}$. Since $2^\lambda=\lambda^+$ in $V[G_{\lambda_0}*g_{[\lambda^+,F(\lambda)]}]$, it follows that $(\lambda^+)^{<\lambda^+}=\lambda^+$, and hence, by the $\Delta$-system lemma, ${\rm domain}(A)$ contains a $\Delta$-system of size $\lambda^{++}$ with root $r$. This produces a contradiction because, in $V[G_{\lambda_0}*g_{[\lambda^+,F(\lambda)]}]$ we have $|2^r|=2^\lambda=\lambda^+$.

To see that $\bar{\Q}$ is $\lambda^{++}$-c.c. in 
\begin{align}
V[G_{\lambda_0}*g_{[\lambda_0,F(\lambda)]}]=V[G_{\lambda_0}*g_{[\lambda^+,F(\lambda)]}][g_{[\lambda_0,\lambda]}]
\label{bigmodel}
\end{align}
we will use the fact that the product of $\theta^+$-Knaster forcing with $\theta^+$-c.c. forcing is $\theta^+$-c.c., where $\theta>\omega$ is a cardinal. Since the forcing $g_{[\lambda_0,\lambda]}\subseteq\Q_{[\lambda_0,\lambda]}$ is $\lambda^{++}$-Knaster and $\bar{\Q}$ is $\lambda^{++}$-c.c. in $V[G_{\lambda_0}*g_{[\lambda^+,F(\lambda)]}]$, it follows that $\bar{\Q}$ is $\lambda^{++}$-c.c. in the model $V[G_{\lambda_0}*g_{[\lambda_0,F(\lambda)]}]=V[G_{\lambda_0}*g_{[\lambda^+,F(\lambda)]}][g_{[\lambda_0,\lambda]}]$.

It remains to show that an antichain of $j_0(\Q_{[\lambda_0,\lambda]})$ in $V[G_{\lambda_0}*g_{[\lambda_0,F(\lambda)]}]$ with size $\lambda^{++}$ would lead to an antichain of $\bar{\Q}$ in $V[G_{\lambda_0}*g_{[\lambda_0,F(\lambda)]}]$ of size $\lambda^{++}$, but this is quite easy. Suppose $A$ is an antichain of $j_0(\Q_{[\lambda_0,\lambda]})$ with size $\delta$ in $V[G_{\lambda_0}*g_{[\lambda_0,F(\lambda)]}]$. Each $p\in A$ is of the form $j_0(h_p)(j"\lambda)$ where $h_p:P_\kappa\lambda\to \Q_{[\lambda_0,\lambda]}$. As mentioned above, each $h_p$ leads to a condition $\bar{h}_p\in\bar{\Q}$. It is easy to check that $\bar{A}:=\{\bar{h}_p\mid p\in A\}$ is an antichain of $\bar{\Q}$ in $V[G_{\lambda_0}*g_{[\lambda_0,F(\lambda)]}]$ of size $\lambda^{++}$.
\end{proof}

\begin{lemma}\label{lemmadonotdisturb}
Forcing with $j_0(\Q_{[\lambda_0,\lambda]})$ over $V[G_{\lambda_0}*g_{[\lambda_0,F(\lambda)]}]$ preserves cardinals and does not disturb the continuum function. 
\end{lemma}

\begin{proof}
By Lemma \ref{lemmadistributive}, $j_0(\Q_{[\lambda_0,\lambda]})$ is ${\leq}\lambda$-distributive in $V[G_{\lambda_0}*g_{[\lambda_0,F(\lambda)]}]$ and thus preserves cardinals in $[\omega,\lambda^+]$ and does not disturb the continuum function on the interval $[\omega,\lambda]$.

Lemma \ref{lemmachaincondition} implies that cardinals in $[\lambda^{++},\infty)$ are preserved. Furthermore, by counting nice names we will now show that the continuum function is not disturbed on $[\lambda^+,\infty)$. Working in $V[G_{\lambda_0}*g_{[\lambda_0,F(\lambda)]}]$, since $j_0(\Q_{[\lambda_0,\lambda]})$ has size at most $| ^\lambda F(\lambda)\cap V[G_{\lambda_0}]|=F(\lambda)$ and is $\lambda^{++}$-c.c., it follows that if $\delta\in[\lambda^+,\infty)$ is a cardinal then there are at most $F(\lambda)^{\lambda^+\cdot\delta}=F(\lambda)^\delta$ nice $j_0(\Q_{[\lambda_0,\lambda]})$-names for subsets of $\delta$. Since $\SCH$ holds in $V[G_{\lambda_0}*g_{[\lambda_0,F(\lambda)]}]$, it follows that for all infinite cardinals $\mu$ and $\nu$, if $\mu\leq 2^\nu$ then $\mu^\nu=2^\nu$ (see \cite[Theorem 5.22(ii)(a)]{Jech:Book}). In particular, we have $F(\lambda)\leq F(\lambda^+)\leq 2^\delta$ and thus $F(\lambda)^\delta=2^\delta$ in $V[G_{\lambda_0}*g_{[\lambda_0,F(\lambda)]}]$. Thus there are at most $2^\delta$ nice $j_0(\Q_{[\lambda_0,\lambda]})$-names for subsets of $\delta$, and the result follows.
\end{proof}


Let $J$ be a $V[G_{\lambda_0}*g_{[\lambda_0,F(\lambda)]}]$-generic filter for $j_0(\Q_{[\lambda_0,\lambda]})$.

\begin{lemma}
$k"J$ generates an $M[j(G_{\lambda_0})]$-generic filter for $j(\Q_{[\lambda_0,\lambda]})$, which we will call $K$.
\end{lemma}

\begin{proof}
Suppose $D\in M[j(G_{\lambda_0})]$ is an open dense subset of $j(\Q_{[\lambda_0,\lambda]})$ and let $D=j(h)(j"\lambda,\alpha)$ for some $h\in V[G_{\lambda_0}]$ with $\dom(h)=P_\kappa\lambda\times\kappa$ and $\alpha<F(\lambda)$. Without loss of generality, let us assume that every element of the range of $h$ is a dense subset of $\Q_{[\lambda_0,\lambda]}$ in $V[G_{\lambda_0}]$. We have $D=j(h)(j"\lambda,\alpha)=k(j_0(h))(j"\lambda,\alpha)$. Define a function $\widetilde{h}\in M_0[j_0(G)]$ with $\dom(\widetilde{h})=\pi(F(\lambda))$ by $\widetilde{h}(\xi)=j_0(h)(j_0"\lambda,\xi)$. Then $\dom(k(\widetilde{h}))=k(\pi(F(\lambda)))=F(\lambda)$ and we have $D=k(\widetilde{h})(\alpha)$. Now the range of $\widetilde{h}$ is a collection of $\pi(F(\lambda))$ open dense subsets of $j_0(\Q_{[\lambda_0,\lambda]})$. Since $j_0(\Q_{[\lambda_0,\lambda]})$ is ${\leq}\pi(F(\lambda))$-distributive in $M_0[j_0(G)]$, one sees that $\widetilde{D}=\bigcap\ran(\widetilde{h})$ is an open dense subset of $j_0(\Q_{[\lambda_0,\lambda]})$. Hence there is a condition $p\in J\cap\widetilde{D}$ and by elementarity, $k(p)\in k"J\cap k(\widetilde{D})\subseteq D$.
\end{proof}

\subsection{Performing surgery}

We will modify the $M[j(G_{\lambda_0})]$-generic filter $K\subseteq j(\Q_{[\lambda_0,\lambda]})$ to get $K^*$ with $j"g_{[\lambda_0,\lambda]}\subseteq K^*$. Then we will argue that $K^*$ remains an $M[j(G_{\lambda_0})]$-generic filter for $j(\Q_{[\lambda_0,\lambda]})$ using the main lemma from \cite{CodyMagidor}.

Let us define $K^*$. Working in $V[G_{\lambda_0}][g_{[\lambda_0,F(\lambda)]}]$, define
\[\dom(j(\Q_{[\lambda_0,\lambda]})):=\bigcup \{\dom(p)\mid p\in j(\Q_{[\lambda_0,\lambda]})\}\]
and let $Q$ be the partial function with $\dom(Q)\subseteq\dom(j(\Q_{[\lambda_0,\lambda]})),$
defined by $Q=\bigcup j"g_{[\lambda_0,\lambda]}$. Given $p\in K$, let $p^*$ be the partial function with $\dom(p^*)=\dom(p)$, obtained from $p$ by altering $p$ on $\dom(p)\cap\dom(Q)$ so that $p^*$ agrees with $Q$. Let
\[K^*=\{p^*\st p\in K\}.\]
Clearly, $j"g_{[\lambda_0,\lambda]}\subseteq K^*$ and it remains to argue that $p^*$ is a condition in $j(\Q_{[\lambda_0,\lambda]})$ for each $p\in K$ and that $K^*$ is an $M[j(G_{\lambda_0})]$-generic filter. This follows from the next lemma, which essentially appears in \cite{CodyMagidor}.

\begin{lemma}\label{lemmasurgery}
Suppose $B\in M[j(G_{\lambda_0})]$ with $B\subseteq j(\dom(\Q_{[\lambda_0,\lambda]}))$ and $|B|^{M[j(G_{\lambda_0})]}\leq j(\lambda)$. Then the set 
$$\mathcal{I}_B=\{\dom(j(q))\cap B\mid \textrm{$q\in \Q_{[\lambda_0,\lambda]}$}\}$$
has size at most $\lambda$ in $V[G_{\lambda_0}*g_{[\lambda_0,F(\lambda)]}]$.
\end{lemma}

\begin{proof}
Let $B$ be as in the statement of the lemma and let $B=j(h)(j"\lambda,\alpha)$ where $h:P_\kappa\lambda^{V}\times\kappa\to P_{\lambda^+}(\dom(\Q_{[\lambda_0,\lambda]}))^{V[G_{\lambda_0}]}$, $\alpha<F(\lambda)$, and $h\in V[G_{\lambda_0}]$. Then $\bigcup\ran(h)$ is a subset of $\dom(\Q_{[\lambda_0,\lambda]})$ in $V[G_{\lambda_0}]$ with $|\bigcup\ran(h)|^{V[G_{\lambda_0}]}\leq\lambda$. Since $V[G_{\lambda_0}]\models \lambda^{<\lambda}=\lambda$ (in $V[G_{\lambda_0}]$ we have $\GCH$ on $[\lambda_0,\lambda]$ and $\lambda$ is a regular cardinal), it will suffice to show that 
$$\mathcal{I}_B\subseteq\{j(d)\cap B\mid d\in P_\lambda(\bigcup\ran(h))^{V[G_{\lambda_0}]}\}.$$
Suppose $\dom(j(q))\cap B\in \mathcal{I}_B$ where $q\in \Q_{[\lambda_0,\lambda]}$. We will show that $\dom(j(q))\cap B=j(d)\cap B$ for some $d\in P_\lambda(\bigcup\ran(h))^{V[G_{\lambda_0}]}$. Let $d:=\dom(q)\cap\bigcup\ran(h)$, then $\dom(j(q))\cap B=j(d)\cap B$ since
$$j(d)=\dom(j(q))\cap\bigcup\ran(j(h))\supseteq \dom(j(q))\cap B.$$
\end{proof}

It now follows from Lemma \ref{lemmasurgery} exactly as in \cite{CodyMagidor} that $K^*$ is an $M[j_0(G_{\lambda_0})]$-generic filter for $j(\Q_{[\lambda_0,\lambda]})$. Now let us show that $K^*\subseteq j(\Q_{[\lambda_0,\lambda]})$. Suppose $p\in j(\Q_{[\lambda_0,\lambda]})$, then since $|\dom(p)|^{M[j(G_{\lambda_0})]}<j(\lambda)$, it follows from Lemma \ref{lemmasurgery}, that the set $\mathcal{I}_{\dom(p)}:=\{\dom(j(q))\cap\dom(p)\st q\in\Q_{[\lambda_0,\lambda]}\}$ has size at most $\lambda$ in $V[G_{\lambda_0}*g_{[\lambda_0,F(\lambda)]}]$. Let $\<I_\alpha\st\alpha<\lambda\>\in V[G_{\lambda_0}*g_{[\lambda_0,F(\lambda)]}]$ be an enumeration of $\mathcal{I}_{\dom(p)}$. By the maximality of the filter $K$, for each $\alpha<\lambda$ we can choose $q_\alpha\in K$ such that $\dom(j(q_\alpha))\cap p=I_\alpha$. It follows that $\<j(q_\alpha)\st\alpha<\lambda\>\in M[j(G_{\lambda_0})]$ because $M[j(G_{\lambda_0})]$ is closed under $\lambda$-sequences in $V[G_{\lambda_0}*g_{[\lambda_0,F(\lambda)]}]$. Since $j(\Q_{[\lambda_0,\lambda]})$ is ${<}j(\lambda_0)$-directed closed, it follows that the partial master condition $m:=\bigcup\{j(q_\alpha):\alpha<\lambda\}$ is a condition in $j(\Q_{[\lambda_0,\lambda]})$, and moreover, $q^*$ can be computed in $M[j(G_{\lambda_0})]$ by comparing $p$ and $m$.

To see that $K^*$ is $M[j(G_{\lambda_0})]$-generic, suppose $A$ is a maximal antichain of $j(\Q_{[\lambda_0,\lambda]})$ in $M[j(G_{\lambda_0})]$. Since $\Q_{[\lambda_0,\lambda]}$ is $\lambda^+$-c.c. in $V[G_{\lambda_0}]$, it follows by elementarity that $\dom(A):=\bigcup\{\dom(r)\st r\in A\}$ has size at most $j(\lambda)$ in $M[j(G_{\lambda_0})]$. Hence by Lemma \ref{lemmasurgery}, we see that $\mathcal{I}_{\dom(A)}:=\{\dom(j(q))\cap \dom(A)\st q\in \Q_{[\lambda_0,\lambda]}\}$ has size at most $\lambda$ in $V[G_{\lambda_0}*g_{[\lambda_0,F(\lambda)]}]$ and is therefore in $M[j(G_{\lambda_0})]$. Using this one can show, as in \cite{CodyMagidor}, that there is a bit-flipping automorphism $\pi_A$ of $j(\Q_{[\lambda_0,\lambda]})$ in $M[j(G_{\lambda_0})]$ such that if $r\in K$ and $\dom(r)\subseteq\dom(A)$ then $\dom(\pi_A(r))=\dom(r)$ and $\pi_A(r)=r^*$. Then since $\pi_A^{-1}[A]\in M[j(G_{\lambda_0})]$ is a maximal antichain of $j(\Q_{[\lambda_0,\lambda]})$, and $K$ is generic for $j(\Q_{[\lambda_0,\lambda]})$ over $M[j(G_{[\lambda_0,\lambda]})]$, it follows that there is a condition $s\in K\cap \pi_A^{-1}[A]$. Then $\pi_A(s)=s^*\in K^*\cap A$, and therefore $K^*$ is generic for $j(\Q_{[\lambda_0,\lambda]})$ over $M[j(G_{\lambda_0})]$.

Thus we may lift the embedding to
\[j:V[G_{\lambda_0}*g_{[\lambda_0,\lambda]}]\to M[j(G_{\lambda_0})*j(g_{[\lambda_0,\lambda]})]\]
where $j(g_{[\lambda_0,\lambda]})=K^*$ and $j$ is a class of $V[G_{\lambda_0}*g_{[\lambda_0,F(\lambda)]}*J]$. It follows that $M[j(G_{\lambda_0})*j(g_{[\lambda_0,\lambda]})]$ is closed under $\lambda$-sequences in $V[G_{\lambda_0}*g_{[\lambda_0,F(\lambda)]}*J]$ and that
\[M[j(G_{\lambda_0})*j(g_{[\lambda_0,\lambda]})]=\{j(h)(j"\lambda,\alpha)\st h:P_\kappa\lambda\times\kappa\to V, \alpha<F(\lambda),h\in V[G_{\lambda_0}*g_{[\lambda_0,\lambda]}]\}.\]

Since the forcing $g_{[\lambda^+,F(\lambda)]}*J\subseteq \Q_{[\lambda^+,F(\lambda)]}*j_0(\Q_{[\lambda_0,\lambda]})$ is ${\leq}\lambda$-distributive in $V[G_{\lambda_0}*g_{[\lambda_0,\lambda]}]$, it follows that the pointwise image $j[g_{[\lambda^+,F(\lambda)]}*J]$ generates an $M[j(G_{\lambda_0})*j(g_{[\lambda_0,\lambda]})]$-generic filter for $j(\Q_{[\lambda^+,F(\lambda)]}*j_0(\Q_{[\lambda_0,\lambda]}))$, denote this filter by $j(g_{[\lambda^+,F(\lambda)]}*J)$. Then the embedding lifts to 
\[j:V[G_{\lambda_0}*g_{[\lambda_0,F(\lambda)]}*J]\to M[j(G_{\lambda_0})*j(g_{[\lambda_0,F(\lambda)]})*j(J)]\]
where $j$ is a class of $V[G_{\lambda_0}*g_{[\lambda_0,F(\lambda)]}*J]$, witnessing that $\kappa$ is $\lambda$-supercompact in this model.

\subsection{Controlling the continuum function at $F(\lambda)^+$ and above}\label{sectionlaststep}

In the model $V[G_{\lambda_0}*g_{[\lambda_0,F(\lambda)]}*J]$ one has $2^\gamma=F(\gamma)$ for every regular cardinal $\gamma\leq F(\lambda)$ and $\GCH$ holds at all cardinals greater than or equal to $F(\lambda)^+$. Working in $V[G_{\lambda_0}*g_{[\lambda_0,F(\lambda)]}*J]$, let $\mathbb{E}$ be the Easton-support product of Cohen forcing
\[\mathbb{E}:=\prod_{\gamma\in[F(\lambda)^+,\infty)\cap\REG}\Add(\gamma,F(\gamma)).\]
Let $E$ be generic for $\mathbb{E}$ over $V[G_{\lambda_0}*g_{[\lambda_0,F(\lambda)]}*J]$. Standard arguments \cite{Easton:PowersOfRegularCardinals} can be used to see that in $V[G_{\lambda_0}*g_{[\lambda_0,F(\lambda)]}*J*E]$, for every regular cardinal $\gamma$ we have $2^\gamma=F(\gamma)$.
Since $\mathbb{E}$ is ${\leq}F(\lambda)$-closed in $V[G_{\lambda_0}*g_{[\lambda_0,F(\lambda)]}*J]$, it follows that the pointwise image $j[E]$ generates an $M[j(G_{\lambda_0})*j(g_{[\lambda_0,F(\lambda)]})*j(J)]$-generic filter for $j(\mathbb{E})$, which we will denote by $j(E)$. Then $j$ lifts to
\[j:V[G_{\lambda_0}*g_{[\lambda_0,F(\lambda)]}*J*E]\to M[j(G_{\lambda_0})*j(g_{[\lambda_0,F(\lambda)]})*j(J)*j(E)]\]
where $j$ is a class of $V[G_{\lambda_0}*g_{[\lambda_0,F(\lambda)]}*J*E]$ witnessing that $\kappa$ is $\lambda$-supercompact in that model.
\end{proof}

\section{Open Questions}

First let us discuss the problem of globally controlling the continuum function on the regular cardinals while preserving multiple instances of partial supercompactness. Suppose $\GCH$ holds and we have regular cardinals $\kappa_0<\eta_0<\kappa_1<\eta_1$ such that for each $\alpha\in\{0,1\}$, $\kappa_\alpha$ is $\eta_\alpha$-supercompact.
Additionally, assume $F$ is a function satisfying the requirements of Easton's theorem (E1) and (E2), and that for each $\alpha$ there is a $j_\alpha:V\to M_\alpha$ with critical point $\kappa_\alpha$ such that $\kappa_\alpha$ is closed under $F$, $M^{\eta_\alpha}\subseteq M$, $H(F(\eta_\alpha))\subseteq M$, and for each regular cardinal $\gamma\leq\eta_\alpha$, $(|j_\alpha(F)(\gamma)|=F(\gamma))^V$. Then, as a corollary to the proof of Theorem \ref{theorem1} above, we obtain the following.
\begin{corollary}\label{corollary}
There is a cardinal preserving forcing extension in which $2^\gamma=F(\gamma)$ for every regular cardinal $\gamma$ and $\kappa_\alpha$ remains $\eta_\alpha$-supercompact for $\alpha\in\{0,1\}$. 
\end{corollary}
This corollary can be obtained by essentially applying the above proof of Theorem \ref{theorem1} twice. For example, first we carry out the proof of Theorem \ref{theorem1} with $\kappa_0$ and $\eta_0$ in place of $\kappa$ and $\lambda$ and where the forcing iteration used terminates before stage $\kappa_1$. Lifting the embedding $j_0:V\to M_0$ witnessing that $\kappa_0$ is $\eta_0$-supercompact requires the ``extra forcing'' that depends on $j_0$. Let $\P_0$ denote the iteration defined so far, including the extra forcing. Since $\P_0$ has size less than the critical point $\kappa_1$ of the next embedding $j_1:V\to M_1$ witnessing the $\eta_1$-supercompactness of $\kappa_1$, it follows by the Levy-Solovay theorem that $j_1$ lifts through the iteration performed so far. Next, working in $V^{\P_0}$, we perform an iteration for controlling the continuum function that picks up where the last one left off. Call the iteration $\P_1$, and lift $j_1$ through the iteration $\P_0*\P_1$ just as we lifted $j_0$ through $\P_0$. Furthermore, since $\P_1$ is highly distributive in $V^{\P_0}$ the first embedding $j_0$ will easily extend to $V^{\P_0*\P_1}$.

Corollary \ref{corollary} only covers a simple configuration of partially supercompact cardinals. Is a more general result possible? It seems that the need for the ``extra forcing'' in our proof of Theorem \ref{theorem1} prevents the method from providing a clear strategy for obtaining a more general result in which more complicated configurations of partially supercompact cardinals are preserved. It may be the case that the uniformity of the Sacks-forcing method, which is applied in \cite{FriedmanHonzik:EastonsTheoremAndLargeCardinals} to obtain analogous global results for measurable as well as strong cardinals, could lead to an answer to Question \ref{questionglobal} below. One would desire a two-cardinal version of Sacks forcing for adding subsets to $\kappa$ that satisfies $\lambda$-fusion.
\begin{question}\label{questionglobal}
Assuming $\GCH$, and given a class of partially supercompact cardinals $S$ and a function $F$ from the class of regular cardinals to the class of cardinals satisfying Easton's requirements (E1) and (E2), under what conditions can one force the continuum function to agree with $F$ at all regular cardinals, while preserving cardinals as well as the full degree of partial supercompactness of each cardinal in $S$?
\end{question}

Another potential way of strengthening Theorem \ref{theorem1} is to weaken the hypothesis. This was done for the analagous theorem concerning measurable cardinals in \cite{FriedmanHonzik:EastonsThmAndLCFromOptimal}. In this direction, we pose the following question.
\begin{question}
Can the hypothesis of Theorem 1 be weakened by replacing the assumption $H(F(\lambda)) \subseteq M$ by the weaker assumption ``$V$ and $M$ have the same cardinals up to and including $F(\lambda)$''? Or, in the special case when $F(\lambda) = \mu^+$ for some regular cardinal $\mu$, by the ostensibly stronger assumption that $H(\mu) \subseteq M$ and $(\mu^+)^M = \mu^+$? (Note however that the latter assumption is actually optimal for the analogous case when one wants to find a model with a measurable cardinal $\kappa$ with $2^\kappa = \mu^+$, where $\mu = \kappa^{+n}$ for some $n>0$; see \cite{FriedmanHonzik:EastonsThmAndLCFromOptimal} for more details.)

\end{question}


\end{document}